\documentclass[a4paper,12pt,reqno]{amsart}
\usepackage{amsmath,amsfonts,amssymb}
\usepackage[style=numeric]{biblatex}
\usepackage{verbatim}
\usepackage{enumerate}
\usepackage{dsfont}
\usepackage{url}
\usepackage[colorlinks]{hyperref}
\usepackage[latin1]{inputenc}
\usepackage{enumitem}
\usepackage[english]{babel}
\usepackage{tikz}
\usepackage[margin=3.5cm]{geometry}
\usepackage{bbm}
\usepackage{multirow}
\usepackage{mathrsfs}
\usepackage{caption}
\usepackage{float}
\restylefloat{table}
\usepackage{graphicx}

\def\R{\mathbb R}
\def\Z{\mathbb Z}
\def\C{\mathbb C}

\def\Xa{\mathcal X}

\def\Fq{\mathbb{F}_q}
\def\Fp{\mathbb{F}_p}

\def\lcm{\mathop{\rm lcm}}

\newtheorem{theorem}{Theorem}[section]
\newtheorem{lemma}[theorem]{Lemma}
\newtheorem{definition}[theorem]{Definition}

\newtheorem{proposition}[theorem]{Proposition}

\theoremstyle{definition}

\author[J. A. Oliveira]{Jos\'e Alves Oliveira}
\address{Departamento de Matem\'{a}tica,
	Universidade Federal de Minas Gerais,
	UFMG,
	Belo Horizonte MG (Brazil),
	30123-970}

\email{joseufmg@gmail.com}

\date{\today
}
\keywords{Finite fields, Gauss Sums, Jacobi Sums, diagonal equations}
\subjclass[2020]{12E20 (primary) and 11T24 (secondary)}

\title{On maximal and minimal hypersurfaces of Fermat type}

\begin{document}

\begin{abstract}
	Let $\Fq$ be a finite field with $q=p^n$ elements. In this paper, we study the number of $\Fq$-rational points on the affine hypersurface $\mathcal X$ given by $a_1 x_1^{d_1}+\dots+a_s x_s^{d_s}=b$, where $b\in\Fq^*$. A classic well-konwn result from Weil yields a bound for such number of points. This paper presents necessary and sufficient conditions for the maximality and minimality of $\mathcal X$ with respect to Weil's bound. 
\end{abstract}

\maketitle

\section{Introduction} 

Let $\Fq$ be a finite field with $q=p^n$ elements, where $p$ is a prime and $n$ is a positive integer. Let $(a_1,\dots,a_s)\in\Fq^s$, $(d_1,\dots,d_s)\in\Z_+^s$ and $b\in\Fq^*$, and let $\Xa\subset \mathbb{A}^s$ be an irreducible Fermat hypersurface defined over $\Fq$ given by
\begin{equation}\label{item500}
\Xa:a_1 x_1^{d_1}+\cdots+a_s x_s^{d_s}=b.
\end{equation}
If $\Xa(\Fq)$ denotes the number of $\Fq$-rational points on $\Xa$ and $b\in\Fq^*$, then the famous Weil bound~\cite{weil1949numbers} entails that
\begin{equation}\label{item502}
\left|\Xa(\Fq)-q^{s-1}\right|\le  q^{(s-2)/2}\left[\sqrt{q}\prod_{i=1}^{s}(d_i-1)-(\sqrt{q}-1)I(d_1,\dots,d_s)\right],
\end{equation}
where $I(d_1,\dots,d_s)$ is the number of $s$-tuples $(y_1,\dots,y_s)\in\Z^s$, with $1\leq y_i\leq d_i-1$ for all $i=1,\dots,s$, such that
\begin{equation}\label{item503}
\frac{y_1}{d_1}+\dots+\frac{y_s}{d_s}\equiv 0\pmod{1}.
\end{equation}
If the number of $\Fq$-rational points on $\Xa$ attains this upper (lower) bound, then it is called a maximal (minimal) hypersurface. In general, maximal hypersurfaces are of great interest due to their wide use for applications in coding theory~\cite{aubry1992reed,stepanov2012codes}. Indeed, the problem of studying the number of points on $\Xa$ has been tackled for many authors. Although the exact number of points on $\Xa$ has been studied in many papers \cite{cao2007factorization,OLIVEIRA2021101927,wolfmann1992number,cao2016number}, the complete characterization of the maximality and minimality of it with respect to Weil's bound has not been provided. The particular case where $\Xa$ is a curve has been studied for some authors \cite{garcia2008certain,tafazolian2013maximal}. For example, maximal and minimal Fermat type curves of the form $x^n+y^m=1$ were studied in \cite{tafazolian2010characterization} and \cite{tafazolian2013maximal}.

More recently, we studied the number of $\Fq$-rational points on $\Xa$ in  a very general setting \cite{OLIVEIRA2021101927}. In particular, we used results from~\cite{shioda1979fermat} regarding Jacobi sums to study the maximality and minimality of $\Xa$ in the case where $d_1=\dots=d_s$. It turned out that the maximal and minimal hypersurfaces given by the Equation~\eqref{item500} satisfy a very natural condition, which is being covered by a Hermitian type hypersurface defined over $\Fq$ given by
$$\mathcal H_r: x_1^{p^r+1}+\dots+x_s^{p^r+1}=1,$$
where $r$ is a positive integer. The aproach used in~\cite{shioda1979fermat} seems not to be applicable in the general situation where $d_1,\dots,d_s$ are distintic, so that we could not characterize the maximality and minimality in the general setting. Our goal in this paper is to provide tools that allow us to study the maximality and minimality of $\Xa$ under some mild conditions. Along the Section~\ref{item508} we provide a new approach to tackle this problem. The main ingredients we use are the celebrated Davemport-Hasse Relation and a purity result for Jacobi sums. In Theorem~\ref{item504} we provide necessary and sufficient conditions in which the number of $\Fq$-rational points on $\Xa$ attains the bound~\eqref{item502}. In particular, we will prove that the condition of being covered by the Hermitian curve remains being necessary in this more general setting. In order to do that, the following definition will be important.
\begin{definition}
	Let $r$ be a positive integer. We say that an integer $d$ is $(p,r)$-admissible if $d\mid (p^r+1)$ and there exists no $r'<r$ such that $d\mid (p^{r'}+1)$.
\end{definition}

Throughout the paper, $s\geq 2$ is an integer, $d_1,\dots,d_s$ are non-negative integers $\ge 2$ dividing $q-1$ and $a_1,\dots,a_s,b$ are elements in $\Fq$ with $b\neq 0$. For $d$ a divisor of $q-1$, let $\chi_d$ be a multiplicative character of order $d$ of $\Fq^*$.

The main result of the paper is the following.

\begin{theorem}\label{item504} Assume that $s\neq 3$ and suppose that $\gcd(d_1,\dots,d_s)>2$ if $s>3$. Then the number of $\Fq$-rational points on $\Xa$ attains the bound \eqref{item502} if and only if
	\begin{itemize}
		\item each $d_i$ is $(p,r)$-admissible;
		\item $\tfrac{n}{2r}$ is even;
		\item $\chi_{d_i}(a_i)=\chi_{d_i}(b)$ for all $i=1,\dots,s$.
	\end{itemize}
Furthermore, $\Xa$ is maximal if $s$ is odd and minimal otherwise.
\end{theorem}

It is worth mentioning that a hypersurface is maximal (or minimal) if and only the inverses of all roots of the associated Zeta function are equal to $-q^{\frac{s-1}{2}}$ and $-q^{\frac{s-2}{2}}$ (or $q^{\frac{s-1}{2}}$ and $q^{\frac{s-2}{2}}$). In particular, in the cases where the conditions of Theorem~\ref{item504} are satisfied, the Zeta function associated to $\Xa$ is provided. Furthermore, our result generalizes many results of the literature, such as the main results of \cite{tafazolian2010characterization,tafazolian2013maximal}, Theorem 4.4 of~\cite{garcia2008certain} and Theorem 2.14 of \cite{OLIVEIRA2021101927}. As we will see in Section~\ref{item508}, the hypotesis imposed here are very important in our approach and are very difficult to ride out. For example, a characterization for maximal and minimal Fermat hypersurfaces in case where $b=0$ remains being an open problem. More results concerning the number of points on these type of hypersurfaces can be found in \cite{Panario}, \cite{OLIVEIRA2021101927} and in the references therein.

\section{Preliminaries}\label{item544}
In this section we provide some well-known definitions and results that will be important in the proof of the main result. Throughout the paper, let $\psi$ be the canonical additive character from $\Fq$ to $\Fp$.

\begin{definition}
	Let $\lambda_1,\dots,\lambda_s$ be a multiplicative characters of $\Fq^*$.
	\begin{enumerate}[label=(\alph*)]
		
		\item  The {\it Gauss sum} of $\lambda_1$ over $\Fq$ is the sum $$G(\lambda_1)=\sum_{x\in\Fq^*}\psi(x)\lambda_1(x).$$
		\item Let $b\in\Fq$. The \textit{Jacobi sum} of $\lambda_1,\ldots, \lambda_s$ is defined as 
		$$J_b(\lambda_1,\ldots,\lambda_s)=\sum_{\substack{b_1+\cdots b_s=b\\ (b_1,\ldots,b_s)\in\Fq^s}}\lambda_1(b_1)\cdots\lambda_s(b_s);$$
	\end{enumerate}
\end{definition}

One can verify that 
$$J_b(\lambda_1,\ldots,\lambda_s)=\lambda_1(b)\dots\lambda_s(b)J_1(\lambda_1,\ldots,\lambda_s)$$
for all $b\in\Fq^*$, fact that will be extensively used in the paper. Throughout the paper, we set $J(\lambda_1,\ldots,\lambda_s)=J_1(\lambda_1,\ldots,\lambda_s)$.

\begin{definition}
	 A Gauss sum  $G(\lambda_1)$ (or a Jacobi sum $J_b(\lambda_1,\ldots,\lambda_s)$) is said to be {\it pure} if some non-zero integral power of it is real.
\end{definition}

\begin{theorem}\cite[Theorem $5.21$]{Lidl}\label{item523}
If $\lambda_1,\dots, \lambda_s$ are nontrivial multiplicative characters of $\Fq^*$, then
$$J(\lambda_1,\dots,\lambda_s)=\begin{cases}
\tfrac{G(\lambda_1)\cdots G(\lambda_s)}{G(\lambda_1,\dots,\lambda_s)},&\text{ if }\lambda_1\cdots\lambda_s\text{ is nontrivial;}\\
-q^{-1}G(\lambda_1)\dots G(\lambda_s),&\text{ if }\lambda_1\cdots\lambda_s\text{ is trivial.}\\
\end{cases}$$
\end{theorem}

\begin{proposition}\cite[Theorem $5.22$]{Lidl}\label{item519} Let $\lambda_1,\ldots,\lambda_s$ be nontrivial multiplicative characters of $\Fq^*$, then 
	$$|J(\lambda_1,\ldots,\lambda_s)|=\begin{cases}
	q^{\frac{s-1}{2}},&\text{ if } \lambda_1\cdots\lambda_s\text{ is nontrivial;}\\
	q^{\frac{s-2}{2}},&\text{ if }\lambda_1\cdots\lambda_s\text{ is trivial.}\\
	\end{cases}$$
\end{proposition}

\begin{lemma}\cite[Theorems 5.11 and 5.12]{Lidl}\label{item501} 
	Let $\chi_0$ denote the trivial multiplicative character of $\Fq^*$. If $\chi\neq\chi_0$, then
	\begin{enumerate}[label=(\alph*)]
		\item\label{item524} $|G(\chi)|=\sqrt{q};$
		\item\label{item522} $G(\chi)G(\overline{\chi})=\chi(-1)q$;
		\item\label{item531} $G(\chi_0)=-1$.
	\end{enumerate}
\end{lemma}

\begin{lemma}\cite[Theorem 5.15]{Lidl}\label{item525}
	Let $q=p^n$ and let $\chi_2$ be the quadratic character of $\Fq^*$. Then
	$$G(\chi_2)=\begin{cases}
	(-1)^{n+1}q^{\frac{1}{2}},&\text{ if }p\equiv 1\pmod{4}\\
	(-1)^{n+1}i^nq^{\frac{1}{2}},&\text{ if }p\equiv 3\pmod{4},\\
	\end{cases}$$
	where $i$ denotes the imaginary unity.
\end{lemma}

\begin{theorem}\cite[Theorem 1]{evans1981pure}\label{item520} 
	Let $q=p^n$. Given a divisor $d>2$ of $q-1$, the following are equivalent:
	\begin{itemize}
		\item $d$ is $(p,r)$-admissible;
		\item $G(\chi_d^j)$ is pure for all $j\in\Z$.
	\end{itemize}
\end{theorem}

\begin{definition}For $d$ a divisor of $q-1$ and $a,b\in\Fq^*$, we set
	$$\theta_d(a,b)=\begin{cases}
	1,&\text{ if }\chi_d(a)=\chi_d(b);\\
	0,&\text{ otherwise.}
	\end{cases}$$
\end{definition}

\begin{theorem}\cite[Theorem 2.11]{OLIVEIRA2021101927}\label{item506} Let $q=p^n$ and $\vec{a}\in\Fq^s$. If each $d_i$ is $(p,r)$-admissible, then the number of $\Fq$-rational points on the affine hypersurface given Equation\eqref{item500} is given by
	\begin{equation}\label{item507}
	q^{s-1}-\varepsilon^{s+1}\sqrt{q}^{s-2}\left(\sqrt{q}\prod_{i=1}^{s}(1-d_i)^{\nu_i(b)}-\sum_{j=1}^{\sqrt{q}-\varepsilon}\prod_{i=1}^{s}(1-d_i)^{\delta_{i,j}}\right)
	\end{equation}
	for $b\neq 0$, where $\delta_{i,j}=\theta_{d_i}(a_i,\alpha^j)$, $\nu_{i}(b)=\theta_{d_i}(a_i,b)$ and $\varepsilon=(-1)^{n/2r}$.
\end{theorem}
Along the proof of our main result we need the following lemma.
\begin{lemma}\label{item539}
	Let $d_1,\dots,d_s$ be positive integers and let $D=\lcm(d_1,\dots,d_s)$. Then
	\begin{equation}\label{item252}
		I(d_1,\dots,d_s)=\tfrac{(-1)^s}{D}\sum_{m=1}^{D}\prod_{d_i|m} (1-d_i).
	\end{equation}
\end{lemma}
\begin{proof}
	A well-known formula for $I(d_1,\dots,d_s)$ (see p.293 of \cite{Lidl}) is the following:
	\begin{equation}\label{item253}
		I(d_1,\dots,d_s)=(-1)^s+(-1)^{s}\sum_{r=1}^s (-1)^{r}\sum_{1\le i_1<\dots<i_r\le s}\frac{d_{i_1}\dots d_{i_r}}{\lcm(d_{i_1},\dots,d_{i_r})}.
	\end{equation}
	Let $i_1,\dots,i_r$ be integers such that $1\le i_1<\dots<i_r\le s$. The product $d_{i_1}\cdots d_{i_r}$ appears in an expansion of a product in \eqref{item252} whenever $d_{i_j}|m$ for all $j=1,\dots,r$, which occurs $\tfrac{D}{\lcm(d_{i_1},\dots,d_{i_r})}$ times, since $1\le m\le D$. Therefore the expressions in \eqref{item252} and \eqref{item253} coincide, proving our result.
\end{proof}

\section{The proof of the main result}\label{item508}
In this section we provide the proof of Theorem~\ref{item504}. In order to do that, the following definitions will be useful.

\begin{definition}\label{item511}
	We define
	$$\Omega=\{z\in\C:\text{ there exists an integer }n\text{ such that }z^n\in\R\}.$$
\end{definition}

Let $\chi,\lambda_1,\dots,\lambda_s$ be a nontrivial multiplicative characters and $b\in\Fq$. We note that $G(\chi)$ is a pure Gauss sum if and only if $G(\chi)\in\Omega$. Also, $J_b(\lambda_1,\dots,\lambda_s)$ is a pure Jacobi sum if and only if $J_b(\lambda_1,\dots,\lambda_s)\in\Omega$. The proof of the following lemma is direct and it will be omitted.

\begin{lemma}\label{item528}
	$(\Omega,\times)$ is a group.
\end{lemma}

\begin{definition} Let $\vec{d}=(d_1,\dots,d_s)\in\Z$ with $d_i\geq 2$ for all $i=1,\dots,s$. We define 
\begin{enumerate}
	\item $\mathcal{B}(\vec{d})=\{(\ell_1,\dots,\ell_s)\in\Z^s:0<\ell_i<d_i\text{ for all }i=1,\dots,s\}$;
	\item $\mathcal{U}(\vec{d})=\{(\ell_1,\dots,\ell_s)\in\mathcal{B}(\vec{d}): \tfrac{y_1}{d_1}+\dots+\tfrac{y_s}{d_s}\not\equiv 0\pmod{1}\};$
	\item $\mathcal{U}^c(\vec{d})=\{(\ell_1,\dots,\ell_s)\in\mathcal{B}(\vec{d}): \tfrac{y_1}{d_1}+\dots+\tfrac{y_s}{d_s}\equiv 0\pmod{1}\}.$
\end{enumerate}
\end{definition}

Let $b\in\Fq^*$. A well known way to compute $\Xa(\Fq)$ is the following:

\begin{equation}\label{item521}
\begin{aligned}\Xa(\Fq)&=\sum_{b_1+\dots+b_s=b}\, \prod_{i=1}^s\left[1+\chi_{d_i}(a_i^{-1}b_i)+\dots+\chi^{d_i-1}_{d_i}(a_i^{-1}b_i)\right]\\
&=q^{s-1}+\sum_{(\ell_1,\dots,\ell_s)\in\mathcal{B}(\vec{d})} \chi_{d_1}^{\ell_1}\big(\tfrac{b}{a_1}\big)\dots \chi_{d_s}^{\ell_s}\big(\tfrac{b}{a_s}\big)J\big(\chi_{d_1}^{\ell_1},\dots,\chi_{d_s}^{\ell_s}\big).\\
\end{aligned}
\end{equation}

Proposition~\ref{item519} states that
$$\big|J\big(\chi_{d_1}^{\ell_1},\dots,\chi_{d_s}^{\ell_s}\big)\big|=\begin{cases}
q^{\frac{s-1}{2}},&\text{ if } (\ell_1,\dots,\ell_s)\in\mathcal{U}(\vec{d});\\
q^{\frac{s-2}{2}},&\text{ if } (\ell_1,\dots,\ell_s)\in\mathcal{U}^c(\vec{d}).\\
\end{cases}$$
From here, we have the following direct result.
\begin{lemma}\label{item536}
	 If $\Xa(\Fq)$ attains the bound \eqref{item502}, then $J\big(\chi_{d_1}^{\ell_1},\dots,\chi_{d_s}^{\ell_s}\big)$ is pure for all $(\ell_1,\dots,\ell_s)\in\mathcal{B}(\vec{d})$.
\end{lemma}
 In what follows, we use this fact to obtain necessary conditions for which bound \eqref{item502} is attained.
The following lemma is a consequence of the well-known Davemport-Hasse Relation and it will be used in the main results of this section.
\begin{lemma}\label{item529} Let $k$ be a positive integer and let $\lambda$ and $\chi$ be multiplicative characters of $\Fq^*$ such that $\lambda$ has order $m\geq 1$. Then
	$$\frac{\prod_{j=0}^{m-1}G\big(\lambda^{kj} \chi\big)}{G\big(\chi^{m/d}\big)^{d}}\in \Omega,$$
	where $d=\gcd(m,k)$.
\end{lemma}

\begin{proof}
	We observe $\lambda^k$ has order $\tfrac{m}{d}$. If $\chi$ also has order $\tfrac{m}{d}$, then item \ref{item522} of Lemma~\ref{item501} and Theorem 5.11 of \cite{Lidl} entail that
	\begin{equation}\label{item526}
	\frac{\prod_{j=0}^{m-1}G\big(\lambda^{kj} \chi\big)}{G\big(\chi^{m/d}\big)^{d}}=\frac{\prod_{j=0}^{m-1}G\big(\lambda^{kj}\big)}{(-1)^{d}}=\begin{cases}
	q^{\frac{m-1}{2}}\prod_{j=1}^{\frac{m-1}{2}}\lambda^j(-1),&\text{ if }m\text{ is odd;}\\
	G(\lambda^{m/2})q^{\frac{m-2}{2}}\prod_{j=1}^{\frac{m-2}{2}}\lambda^j(-1),&\text{ if }m\text{ is even.}\\
	\end{cases}
	\end{equation}
	The assertion follows by Equation~\eqref{item526} and Lemma~\ref{item525}. If $d=m$, then
	$$\prod_{j=0}^{m-1}G\big(\lambda^{kj} \chi\big)=\prod_{j=0}^{m-1}G\big( \chi\big)=G\big( \chi\big)^d,$$
	which proves the assertion. Now, we assume that the order of $\chi$ is not $\tfrac{m}{d}$ and $d\neq m$. It follows by Lemma~\ref{item525} and Corollary $5.29$ of \cite{Lidl} that
	$$\frac{\prod_{j=0}^{m/d-1}G\big(\lambda^{kj} \chi\big)}{G\big(\chi^{m/d}\big)}\in \Omega,$$
	and therefore our result follows by using that $\lambda^k$ has order $\tfrac{m}{d}$.
\end{proof}

The following technical results will be useful.

\begin{lemma}\cite[Theorem]{sun1987solvability}\label{item240}
	Let $s>2$ be an integer and let $d_1,\dots,d_s$ be positive integers. Then $I(d_1,\dots,d_s)=0$ if and only if one of the following holds:
	\begin{itemize}
		\item for some $d_i$, $\gcd(d_i,d_1\dots d_s/d_i)=1$, or
		\item if $d_{i_1},\dots,d_{i_k}$ $(1\le i_1<\dots i_k\le s)$ is the set of all even integers among $\{d_1,\dots,d_s\}$, then $2\nmid k$, $d_{i_1}/2,\dots,d_{i_k}/2$ are pairwise prime, and $d_{i_j}$is prime to any odd number in $\{d_1,\dots,d_s\}\, (j=1,\dots,k)$.
	\end{itemize}
	
\end{lemma}

The following proposition is the key step of the proof of our main results. 

\begin{proposition}\label{item534}
	Let $s\neq 3$ be an integer and suppose $\gcd(d_1,\dots,d_s)>2$ if $s>3$. If $J\big(\chi_{d_1}^{\ell_1},\dots,\chi_{d_s}^{\ell_s}\big)$ is pure for all $(\ell_1,\dots,\ell_s)\in\mathcal{B}(\vec{d})$, then $G\big(\chi_{d_j}^\ell\big)$ is pure for all $\ell\in\Z$ and $j=1,\dots,s$.
\end{proposition}

\begin{proof}
	We will prove that $G\big(\chi_{d_1}^\ell\big)$ is pure for all $\ell\in\Z$. If $\ell\equiv 0\pmod{d_1}$, then the result follows directly by item \ref{item531} of Lemma~\ref{item501}. Assume that $\ell\not\equiv 0\pmod{d_1}$. Since $\gcd(d_1,\dots,d_s)>2$ and $s\neq 3$, it follows from Lemma~\ref{item240} that there exist a $(s-2)$-tuple $(m_3,\dots,m_{s})\in\Z^{s-2}$, with $1\leq m_i\leq d_i-1$ for all $i=3,\dots,s$, such that
		$$\frac{m_3}{d_3}+\dots+\frac{m_{s}}{d_{s}}\equiv 0\pmod{1}.$$
Then, by Theorem~\ref{item523} and  Lemmas~\ref{item501} and \ref{item528}, we have that
\begin{equation}\label{item527}
	\tfrac{G\big(\chi_{d_{1}}^{\ell_{1}}\big)^2 G\big(\chi_{2}^{\ell_2}\big)^2 }{G\big(\chi_{d_{1}}^{\ell_{1}}\chi_{d_{2}}^{\ell_{2}}\big)^2}=\tfrac{J\big(\chi_{d_1}^{\ell_1},\chi_{d_{2}}^{\ell_{2}},\chi_{d_{2}}^{m_{3}},\dots,\chi_{d_s}^{m_s}\big)J\big(\chi_{d_1}^{\ell_1},\chi_{d_{2}}^{\ell_{2}},\chi_{d_{3}}^{-m_{3}},\dots,\chi_{d_s}^{-m_s}\big)}{q^{\frac{s-2}{2}}\chi_{d_3}(-1)\dots\chi_{d_s}(-1)}\in\Omega
\end{equation}
for all $1\leq \ell_{1}< d_{1}$ and $1\leq \ell_2< d_2$. By Lemmas~~\ref{item501}, \ref{item525}, \ref{item528} and \ref{item529}, we have that
$$\prod_{\ell_2=1}^{d_2-1}\tfrac{G\big(\chi_{d_{1}}^{\ell_{1}}\big)^2 G\big(\chi_{2}^{\ell_2}\big)^2}{G\big(\chi_{d_{1}}^{\ell_{1}}\chi_{d_{2}}^{\ell_{2}}\big)^2}=\lambda\tfrac{G\big(\chi_{d_{1}}^{\ell_{1}}\big)^{2d_2}}{G\big(\chi_{ d_{1}}^{\ell_1 d_2}\big)^2}$$
where $\lambda\in\Omega$. In particular, it follows from Lemma~\ref{item529} that
\begin{equation}\label{item530}
	\tfrac{G\big(\chi_{d_{1}}^{\ell_{1}}\big)^{d_2}}{G\big(\chi_{d_{1}}^{\ell_1 d_2}\big)}\in\Omega
\end{equation}
for all $1\leq \ell_{1}< d_{1}$. Now we fix an element $\ell=1,\dots,d_1-1$ in order to prove that $G(\chi_{d-1}^{\ell})$ is pure. We split the proof into two cases:
\begin{itemize}
	\item There exists an integer $u\geq 1$ such that $\ell d_2^u\equiv 0\pmod{d_1}$. In this case, we employ Equation~\eqref{item530}, Lemma~\ref{item528} and item \ref{item531} of Lemma~\ref{item501} and obtain
	$$-G\big(\chi_{d_{1}}^{\ell}\big)=\prod_{j=0}^{u-1} \tfrac{G\Big(\chi_{d_{1}}^{\ell d_2^j}\Big)^{d_2^{-j}}}{G\Big(\chi_{d_{1}}^{\ell d_2^{j+1}}\Big)^{d_2^{-(j+1)}}}\in\Omega.$$
	\item There exist integer $u> 0$ and $v>u $ such that $\ell d_2^v\equiv \ell d_2^u\pmod{d_1}$. Ewe employ Equation~\eqref{item530} and Lemma~\ref{item528} in order to obtain
	\begin{equation}\label{item532}
		G\Big(\chi_{d_{1}}^{\ell d_2^u}\Big)^{d_2^{-u}-d_2^{-v}}=\tfrac{G\Big(\chi_{d_{1}}^{\ell d_2^u}\Big)^{d_2^{-u}}}{G\Big(\chi_{d_{1}}^{\ell d_2^{v}}\Big)^{d_2^{-v}}}=\prod_{j=u}^{v-1} \tfrac{G\Big(\chi_{d_{1}}^{\ell d_2^j}\Big)^{d_2^{-j}}}{G\Big(\chi_{d_{1}}^{\ell d_2^{j+1}}\Big)^{d_2^{-(j+1)}}}\in\Omega.
	\end{equation}
	Furthermore, Equation~\eqref{item530} and Lemma~\ref{item528} entail that
	\begin{equation}\label{item533}
	\tfrac{G\big(\chi_{d_{1}}^{\ell}\big)}{G\Big(\chi_{d_{1}}^{\ell d_2^{u}}\Big)^{d_2^{-(u)}}}=\prod_{j=0}^{u-1} \tfrac{G\Big(\chi_{d_{1}}^{\ell d_2^j}\Big)^{d_2^{-j}}}{G\Big(\chi_{d_{1}}^{\ell d_2^{j+1}}\Big)^{d_2^{-(j+1)}}}\in\Omega.
	\end{equation}
Lemma~\ref{item528} and Equations \eqref{item532} and \eqref{item533} imply that $G\big(\chi_{d_{1}}^{\ell}\big)\in\Omega$. 
\end{itemize}
The cases where $i=2,\dots,s$ follow similarly.
\end{proof}

The following theorem is one of the most important results of the paper.
\begin{theorem}\label{item537}
	Let $s\neq 3$ be an integer and suppose $\gcd(d_1,\dots,d_s)>2$ if $s>3$. If $J\big(\chi_{d_1}^{\ell_1},\dots,\chi_{d_s}^{\ell_s}\big)$ is pure for all $(\ell_1,\dots,\ell_s)\in\mathcal{B}(\vec{d})$, then each $d_i$ is $(p,r)$-admissible.
\end{theorem}

\begin{proof}
	By Theorem~\ref{item520} and Proposition~\ref{item534}, it follows that for each $i=1,\dots,s$ there exists an positive integer $r_i$ such that $d_i$ is $(p,r_i)$-admissible. Since $\gcd(d_1,\dots,d_s)>2$ and $s\neq 3$, it follows from Lemma~\ref{item240} that there exist a $(s-2)$-tuple $(m_3,\dots,m_{s})\in\Z^{s-2}$, with $1\leq m_i\leq d_i-1$ for all $i=3,\dots,s$, such that
	$$\frac{m_3}{d_3}+\dots+\frac{m_{s}}{d_{s}}\equiv 0\pmod{1}.$$
	By Theorem~\ref{item523}, we have that
	$$\tfrac{G\big(\chi_{d_1}^{\ell_1}\big)G\big(\chi_{d_2}^{\ell_2}\big)G\big(\chi_{d_3}^{m_3}\big)\dots G\big(\chi_{d_s}^{m_s}\big)}{G\big(\chi_{d_1}^{\ell_1}\chi_{d_2}^{\ell_2}\big)}=J\big(\chi_{d_1}^{\ell_1},\chi_{d_{2}}^{\ell_{2}},\chi_{d_{2}}^{m_{3}},\dots,\chi_{d_s}^{m_s}\big)\in\Omega$$
	for all $1\leq \ell_1<d_1$ and $1\leq \ell_2<d_2$ and then, by Lemma~\ref{item528} and Proposition~\ref{item534}, it follows that 
	\begin{equation}\label{item535}
		G\big(\chi_{d_1}^{\ell_1}\chi_{d_2}^{\ell_2}\big)\in\Omega
	\end{equation}
	 for all $1\leq \ell_1<d_1$ and $1\leq \ell_2<d_2$. We set $D=\lcm(d_1,d_2)$. Let $\ell_1$ and $\ell_2$ be integers such that $$\big(\tfrac{q-1}{d_1}\big)\ell_1+\big(\tfrac{q-1}{d_2}\big)\ell_2=\gcd\big(\tfrac{q-1}{d_1},\tfrac{q-1}{d_2}\big)=\tfrac{q-1}{D}.$$
	We will prove the following claim.
	
	\textit{Claim:} $G\big(\chi_{d_1}^{t \ell_1}\chi_{d_2}^{t \ell_2}\big)=G\big(\chi_{D}^t)$ is pure for all $t\in\Z$. 
	
	\textit{Proof of the claim:} If $t\equiv 0\pmod{d_1}$ or $t\equiv 0\pmod{d_2}$, then the claim follows by Proposition~\ref{item534}. Assume that $t\not\equiv 0\pmod{d_1}$ or $t\not\equiv 0\pmod{d_2}$. In this case, Equation~\eqref{item535} states that $G\big(\chi_{D}^t)$ is pure, which completes the proof of the claim.
	
	We observe that the claim and Theorem~\ref{item520} imply that there exists an integer $r$ such that $D$ is $(p,r)$-admissible. In particular, it follows that $r_1=r_2=r$. By using the same arguments, we prove that $r_1=\dots=r_s$ and our result follows.
\end{proof}

With the results obtained in this section, we are able to complete the proof of our main result.

\subsection{Proof of Theorem~\ref{item504}}  Assume that the bound \eqref{item502} is attained, then Lemma~\ref{item536} and Theorem~\ref{item537} state that each $d_i$ is $(p,r)$-admissible. In particular, since $d_i|(p^n-1)$, it follows that $2r|n$. Therefore, we are under the hypothesis of Theorem~\ref{item506}, which states that the bound \eqref{item502} is attained if and only if
\begin{equation}\label{item538}
	\prod_{i=1}^{s}(d_i-1)^{\nu_i(b)}=\prod_{i=1}^{s}(d_i-1),
\end{equation}
where $\nu_{i}(b)=\theta_{d_i}(a_i,b)$. Therefore, by definition of $\theta_{d_i}(a_i,b)$ and Equation~\eqref{item538}, we have that $\chi_{d_i}(a_i)=\chi_{d_i}(b)$ for all $i=1,\dots,s$.
 
For the converse, assume that $d_i$ is $(p,r)$-admissible (and then $n$ is even) and suppose that $\chi_{d_i}(a_i)=\chi_{d_i}(b)$ for all $i=1,\dots,s$. Since $\chi_{d_i}(a_i)=\chi_{d_i}(b)$, there exists $c_i\in\Fq$ such that $\tfrac{a_i}{b}=c_i^{d_i}$. Therefore, we can make a change of variables by replacing $c_i x_i$ by $ y_i$ in Equation~\eqref{item500} so that $\Xa(\Fq)$ equals the number of solutions of the equation
$$y_1^{d_1}+\cdots+y_s^{d_s}=1.$$
Hence, by Theorem~\ref{item506}, we have that
\begin{equation}\label{item541}
	\begin{aligned}
		\Xa(\Fq)&=q^{s-1}-\varepsilon^{s+1}\sqrt{q}^{s-2}\left(\sqrt{q}\prod_{i=1}^{s}(1-d_i)-\sum_{j=1}^{\sqrt{q}-\varepsilon}\prod_{i=1}^{s}(1-d_i)^{\delta_{i,j}}\right)\\
		&=q^{s-1}-\varepsilon^{s+1}\sqrt{q}^{s-2}\left(\sqrt{q}\prod_{i=1}^{s}(1-d_i)-\tfrac{(\sqrt{q}-\varepsilon)}{\sqrt{q}-\varepsilon}\sum_{m=0}^{\sqrt{q}-\varepsilon}\prod_{d_i|m} (1-d_i)\right)\\
		&=q^{s-1}-\varepsilon^{s+1}\sqrt{q}^{s-2}(-1)^s\left(\sqrt{q}\prod_{i=1}^{s}(d_i-1)-(\sqrt{q}-\varepsilon)I(d_1,\dots,d_s)\right),\\
	\end{aligned}
\end{equation}
where the last equality follows by Lemma~\ref{item539}. From here, we observe that it is necessary and sufficient that $(-1)^{\frac{n}{2r}}=\varepsilon=1$, which it is equivalent to $\tfrac{n}{2r}$ being even. Therefore the maximality and minimality depend only on the sign $-(-1)^s$, which completes the proof. $\hfill\qed$

\printbibliography

\end{document}